\documentclass[a4paper,11pt]{amsart}
\usepackage{graphicx,cite,amssymb,amscd,amsthm,latexsym,epsfig,amsmath,enumerate,xcolor,mathrsfs}  
\usepackage{comment}
\usepackage{esint}
\usepackage[shortlabels]{enumitem}

\usepackage[colorlinks=true, allcolors=blue]{hyperref}
\usepackage{amsrefs}
\newtheorem{theorem}{Theorem}[section]

\newtheorem{corollary}[theorem]{Corollary}
\newtheorem{lemma}[theorem]{Lemma}

\theoremstyle{definition}
\newtheorem{definition}[theorem]{Definition}

\newcommand{\re}{\mathbb{R}}

\newcommand{\D}{\mathrm{div }}

\numberwithin{equation}{section}

\allowdisplaybreaks

\author[T.~M.~Nascimento]{Thialita M.~Nascimento}
\address[T.~M.~Nascimento]{Department of Mathematics \\ University of Florida \\ P.O. Box 118105 \\ Gainesville \\ FL 32611 \\ USA}
\email{thnasc@iastate.edu}

\author[X.~H.~Nguyen]{Xuan Hien Nguyen}
\address[X.~H.~Nguyen]{Department of Mathematics \\ Iowa State University \\ 396 Carver Hall \\ Ames \\ IA 50010 \\ USA}
\email{xhnguyen@iastate.edu}

\author[P.~R.~Stinga]{Pablo Ra\'ul Stinga}
\address[P.~R.~Stinga]{Department of Mathematics \\ Iowa State University \\ 396 Carver Hall \\ Ames \\ IA 50010 \\ USA}
\email{stinga@iastate.edu}


\thanks{The third author was partially supported by Simons Foundation grant MP-TSM-00002709}

\begin{document}
%%%%%%%%%%%%%%%%%%%%%%%%%%%%%%%%%%%%%%%%%%%%%%%%%%

%%%%%%%%%%%%%%%%%%%%%%%%%%%%%%%%%%%%%%%%%%%%%%%%%%
\title[Regularity of degenerate or singular free boundary problem]{Regularity of solutions
to degenerate and singular free boundary problems with volume constraint}
%%%%%%%%%%%%%%%%%%%%%%%%%%%%%%%%%%%%%%%%%%%%%%%%%%

%%%%%%%%%%%%%%%%%%%%%%%%%%%%%%%%%%%%%%%%%%%%%%%%%%
\begin{abstract}
We prove existence and regularity of solutions to degenerate and singular elliptic free boundary problems,
where the volume of the positivity set of the solution is prescribed.
\end{abstract}
%%%%%%%%%%%%%%%%%%%%%%%%%%%%%%%%%%%%%%%%%%%%%%%%%%

\maketitle

%%%%%%%%%%%%%%%%%%%%%%%%%%%%%%%%%%%%%%%%%%%%%%%%%%
\section{Introduction}
%%%%%%%%%%%%%%%%%%%%%%%%%%%%%%%%%%%%%%%%%%%%%%%%%%

 Our aim is to show existence and regularity of solutions $u$ to the following weighted volume
constrained free boundary problem, where the measure of the positivity set of $u$
is prescribed. Let $\omega=\omega(x)$ be a Muckenhoupt $A_2$ weight on $\re^n$, $n\geq2$.
For a measurable set $E$ we denote $\omega(E)=\int_E\omega(x)\,dx$.
Fix $0<m<\omega(B_1)$ and a smooth function $g$ on $\overline{B_1}$
such that $g\geq c > 0$ on $\partial B_1$. We consider the weighted minimization problem
\begin{equation}\label{weighted prob}
\min_{v\in K_0}\int_{B_1}|\nabla v|^2\omega\,dx
\end{equation}
where
$$
K_0 = \big\{ v \in H^1 ( B_1, \omega) :   \,  v\big|_{\partial B_1} = g, \,  \omega (\{ v > 0 \}\cap B_1) = m\big\}$$
and $H^1(B_1,\omega)$ denotes the corresponding weighted Sobolev space (see Section \ref{prelim sct} for notation).

Free boundary problems with volume constraints as \eqref{weighted prob}
have received significant attention over the past five decades
owing to their crucial role in applied sciences and the mathematical challenges they present
(see, for example, \cite{AguiAltCaff86,ACS87,BMW06,Teix05,Teix10}).
These problems arise in the study of best insulation devices,
semiconductor theory and plasma physics, just to name a few.
From a mathematical perspective, the questions of existence of an optimal configuration $u$,
as well as regularity properties of $u$
and of the free boundary $\partial \{ u > 0\}$, pose important difficulties.
The first, most influential work in this direction is the paper by Aguilera--Alt--Caffarelli \cite{AguiAltCaff86},
in which they consider \eqref{weighted prob} with no weight, that is, $\omega\equiv1$.
In this case, they reduce the problem to a one phase free boundary problem and apply the results of Alt--Caffarelli \cite{ALtCaff81}
to prove existence and regularity of solutions.
Generalizations of the unweighted case were carried out in, for example, \cite{Led96, AFMT99, Tilli99}.
More recently, nonlocal versions of the unweighted problem were considered in, for instance,
\cite{TeixTeym, DaSilvaRossi19}. In particular, in \cite{TeixTeym}, the Caffarelli--Silvestre extension problem
for the fractional Laplacian in $\re^n$ (see, for instance, \cite[Part~IV]{Stinga_2024})
leads the authors to consider weighted equations of the form $\D (\omega (X) \nabla u ) = f(X)$,
but the problem is different than \eqref{weighted prob}.
Our work is thus motivated by those of Aguilera--Alt--Caffarelli \cite{AguiAltCaff86} and
Teixeira--Teymurazyan \cite{TeixTeym}.

From the mathematical viewpoint, this article is the first (up to the best of our knowledge) to establish the existence of an open optimal set $\{u> 0\}$ with prescribed volume in weighted spaces, and the regularity of $u$. Our problem \eqref{weighted prob} presents at least two challenges compared to other optimal design problems governed by elliptic equations. The first one concerns the local regularity of minimizers. Since $A_2$ weights are, in general, only measurable, locally integrable functions, one can only expect H\"older continuity for a minimizer $u$. This limitation prevents us from applying the ideas and techniques used in \cite{AguiAltCaff86}, \cite{ALtCaff81} or \cite{TeixTeym} to prove existence and regularity,
where the Lipschitz regularity of minimizers up to the free boundary
$\partial \{ u > 0 \}$ plays a fundamental role. The second difficulty is the absence of an
isoperimetric inequality for general $A_2$ weights,
which is an essential tool in our approach for adjusting the measure
of the positivity set to the desired value in order to prove existence.
To be more precise, in the proof of existence, we need to assume that $\omega\in A_2$ satisfies either
one of the following two conditions.
\begin{itemize}
\item \textbf{Weighted isoperimetric inequality.}~There is a constant $C_0>0$ such that,
for any $u\in W^{1,1}(B_1)$ and $s>0$,
\begin{equation}\label{eq:isoperimetricassumption}
\omega(\{u>s\}\cap B_1)^{\frac{n-1}{n}}\leq C_0\int_{\{u=s\}\cap B_1}\omega\,d\mathcal{H}^{n-1}
\end{equation}
where $d\mathcal{H}^{n-1}$ denotes the $(n-1)$-dimensional Hausdorff measure in $\mathbb{R}^n$.
\item \textbf{Bound from below.}~There is $\tau>0$ such that
\begin{equation}\label{eq:lowerboundassumption}
\omega(x)\geq\tau\qquad\hbox{for a.e.}~x\in B_1.
\end{equation}
\end{itemize}
However, only the $A_2$ condition on $\omega$ is used to prove regularity.

Our main result reads as follows (see Section \ref{prelim sct} for notation).

\begin{theorem}[Existence and regularity of solutions]\label{thm:main}
Let $\omega$ be an $A_2$ weight. Assume that either $\omega$ satisfies
the isoperimetric inequality \eqref{eq:isoperimetricassumption},
or that it is bounded below away from zero in $B_1$ as in \eqref{eq:lowerboundassumption}. Then there is a minimizer $u\in K_0$ 
to problem \eqref{weighted prob}
such that $0\leq u\leq\|g\|_{L^\infty(\partial B_1)}$ in $B_1$. Moreover, $u$ is locally H\"older continuous in $B_1$, that is,
there exists $0<\beta<1$ depending only on $[\omega]_{A_2}$ and $n$ such that, for any compact set $K\subset B_1$,
$$[u]_{C^{\beta}(K)}\leq C\|u\|_{L^2(B_1,\omega)}$$
where $C>0$ depends only on $K$, $[\omega]_{A_2}$ and $n$. Furthermore, $u$
is a weak solution to $\D(\omega(x)\nabla u)=0$ in the open set $\{ u > 0\}\cap B_1$. 
\end{theorem}

It is an open problem to characterize the weights $\omega$ for which \eqref{eq:isoperimetricassumption}
is true and, in particular, whether or not is true for general $A_2$ weights.

Examples of $A_2$ weights for which \eqref{eq:isoperimetricassumption} holds
are $A_1$ weights, that is, weights $\omega$ for which there is a constant $C>0$ such that
$M\omega(x)\leq C\omega(x)$ for a.e.~$x\in\mathbb{R}^n$, where $M$ is the classical Hardy--Littlewood
maximal function (note that $A_1\subset A_2$). Indeed, this is a consequence of the results in \cite[Section~2.4.2]{Turesson00}.
In fact, \eqref{eq:isoperimetricassumption} is stated as \cite[Theorem~2.4.8]{Turesson00}
(this result is valid for our case because Theorem 2.4.7--the co-area formula, and Lemma 2.4.1--the Poincar\'e inequality, in \cite{Turesson00}
are true under our hypotheses as well). Furthermore, we do not require the $A_1$ weight to be continuous.
Canonical examples of $A_2$ weights satisfying \eqref{eq:lowerboundassumption}
are $\omega(x)=|x|^{-\alpha}$ for $0<\alpha<n$, or, more generally,
$\omega(x)=\operatorname{dist}(x,Z)^{- \alpha}$, where  $Z$ is a ``thin" set of zero Lebesgue measure,
and anisotropic variations of power weights
(see \cite{LST} where the lower bound in $\omega$ is assumed to study a weighted one phase problem
generalizing \cite{ALtCaff81}). In particular, for the case of the extension problem for the fractional Laplacian $(-\Delta)^s$,
where $\omega(x,y)=|y|^{1-2s}$ for $(x,y)\in\mathbb{R}^{n+1}$, we have the restriction $1/2\leq s<1$.

Our approach to the proof of Theorem \ref{thm:main} relies on a penalization method 
that allows us to study the problem without relying on the smoothness of the free boundary as in \cite{AguiAltCaff86}.
The new functional to be minimized takes the form 
$$\int_{B_1} |\nabla v |^2\omega\,dx + f_{\varepsilon} (v)$$
for some parameter $\varepsilon>0$. Assuming only the $A_2$ condition on the weight,
we show that minimizers of the penalized problem exist for all $\varepsilon > 0$ and are H\"older continuous.
The H\"older continuity is proved using a blow-up approximation
by constants in the spirit of Caffarelli's regularity technique. 
Notice that this technique was also implemented in \cite{LST} for the one-phase problem.
However, unlike \cite{LST}, we are able to establish H\"older continuity of solutions to our problem \eqref{weighted prob}
without imposing any further restrictions on the weight $\omega$ (other than being in the $A_2$ class).
Then we show that there exists $\varepsilon_\ast> 0$ small enough such that $u_\ast$ is a minimizer 
for the penalized problem if and only if $u_\ast$ is a minimizer for \eqref{weighted prob}.
Here we argue by contradiction, finding an appropriate competitor function $u_t$ for which
$\{ u_t > 0\}$ approximates in measure $ \{ u > 0\} $, and we apply the hypothesis on $\omega$
stated in Theorem \ref{thm:main}. This argument for proving existence of minimizers was implemented in \cite{BHP05} in the unweighted case, where Lipschitz continuity of minimizers was also proved.

The analysis presented in this paper extends without difficulty to volume constrained free boundary problems
involving more general energies of the form
$$\int_{\Omega}(\nabla v)^TA(x) \nabla v\, dx$$
where $\Omega$ is a smooth bounded domain
and the matrix of coefficients $A(x)$ is bounded, measurable, and 
degenerate elliptic with respect to the weight $\omega$, that is,
there are constants $0<\lambda\leq\Lambda$ such that
$\lambda\omega(x) |\xi|^2 \le\xi^TA(x)\xi \le \Lambda\omega(x) |\xi|^2$ for a.e.\,$x\in\Omega$
and all $\xi\in\re^n$. In this case, the constants in Theorem \ref{thm:main}
will also depend on $\lambda$ and $\Lambda$. This general case is particularly interesting
because it arises when considering fractional elliptic equations in divergence form
$(-\D(a(x)\nabla))^sw=f$ in $\Omega$, see \cite{Caffarelli-Stinga,Stinga-Zhang}.

An interesting question is whether existence of minimizers fails if one removes
either assumption \eqref{eq:isoperimetricassumption} or \eqref{eq:lowerboundassumption}.
We believe that one should be able to remove them, but the technique will have to be different.

The paper is organized as follows. In Section \ref{prelim sct}, we introduce preliminary results about weighted Sobolev spaces
and degenerate/singular elliptic equations. Section \ref{general prop sct} is dedicated to the analysis of the penalized problem, including existence, boundedness and local regularity. These properties are derived without any additional assumptions on the $A_2$ weight.
As mentioned earlier, to adjust the measure of the positivity set to the prescribed value
and close the existence part of Theorem \ref{thm:main}, we must consider weights for which either an isoperimetric inequality
holds or that are bounded below away from zero in $B_1$.
Thus, in Section \ref{small eps sct}, under such assumptions, we finalize the proof of our main result.

%%%%%%%%%%%%%%%%%%%%%%%%%%%%%%%%%%%%%%%%%%%%%%%%%%
\section{Preliminaries}\label{prelim sct}
%%%%%%%%%%%%%%%%%%%%%%%%%%%%%%%%%%%%%%%%%%%%%%%%%%

In this section we collect results regarding weighted Sobolev spaces and degenerate elliptic equations.
References here are \cite{FKS,HKM,Turesson00}, see also \cite{FJK,FKJ}.

We denote by $B_R(x)$ the ball in $\mathbb R^n$ of radius $R$ centered at $x$. When the center is the origin, we omit it and just write $B_R$. 

A nonnegative measurable function $\omega=\omega(x)$ defined on $\re^n$ is an
$A_2$ weight if both $\omega$ and $\omega^{-1}$ are locally integrable and 
\begin{equation}\label{A_2 w definiton}
[\omega]_{A_2}:=\sup\limits_{B \subset\re^n} \left( \frac{1}{|B|} \int_{B} \omega\, dx \right) \left( \frac{1}{|B|} \int_B \omega^{-1}\,dx \right) <\infty
\end{equation}
where the supremum is taken over all balls $B\subset \re^n$. Here $|E|$ denotes the Lebesgue measure of
a measurable set $E \subset \mathbb{R}^n$. 
In this case, we call \eqref{A_2 w definiton} the $A_2$ condition and write $\omega\in A_2$.

Throughout this section, we assume that $\omega\in A_2$.

By definition of $A_2$ weights, we may assume that $0<\omega<\infty$ almost everywhere. Moreover, as $\omega\in L^1_{\mathrm{loc}}(\mathbb{R}^n)$, the induced measure defined by
\[
\omega(E):=\int_E \omega(x)\,dx
\]
is a positive Radon measure. Moreover, the weight $\tilde{\omega}(x):=\omega(x_0+rx)$, for fixed $x_0\in\re^n$ and $r>0$,
is also an $A_2$ weight with $[\tilde{\omega}]_{A_2}=[\omega]_{A_2}$.

\begin{lemma}[Properties of $A_2$ weights]\label{lem:propertiesofA2}
Let $\omega$ be an $A_2$ weight on $\re^n$.
\begin{enumerate}[$(1)$]
\item There exists a constant $c_1>0$ depending only on $[\omega]_{A_2}$ such that 
$$c_{1} \left( \frac{|E|}{|B|} \right)^2 \le  \frac{\omega(E)}{\omega(B)}$$
whenever $B \subset \mathbb{R}^n$ is a ball and $E\subset B$ is a measurable set.
\item There exists a \emph{doubling constant} $D=D(n,[\omega]_{A_2})> 0$ such that 
$$0 < \omega( B_{2r} (x_0)) \le D \omega (B_r(x_0))$$
for all balls $B_r(x_0) \subset \mathbb{R}^n$. 
\item \label{Reverse strong doubling}There exist constants $C_2 > 1$ and $0<S_D\leq1$
depending only on $n$ and $D$ such that 
$$\frac{\omega(E)}{\omega(B)} \le C_2  \left( \frac{|E|}{|B|} \right)^{S_D}$$
whenever $B \subset \mathbb{R}^n$ is a ball and $E\subset B$ is a measurable set.
\end{enumerate}
\end{lemma}

By Lemma \ref{lem:propertiesofA2}$(1)$, we get that $\omega(E)=0$ if and only if $|E|=0$.

The set $L^2(B_1,\omega)$ is the Hilbert space of all measurable functions $u$
on $B_1$ such that
$$\|u\|_{L^2(B_1,\omega)}^2=\int_{B_1}|u|^2\omega\,dx<\infty.$$
By H\"older's inequality, $L^2(B_1,\omega)\subset L^1(B_1)$.
The weighted Sobolev space $H^1(B_1,\omega)$ is the Hilbert space of functions
$u\in L^2(B_1,\omega)$ such that the weak gradient $\nabla u$ is in $L^2(B_1,\omega)$ under the obvious norm.
We have that $C^\infty(\overline{B_1})$ is dense in $H^1(B_1,\omega)$.
Using standard arguments as in the unweighted case,
one can show that the embedding $H^1(B_1,\omega)\subset\subset L^2(B_1,\omega)$ is compact.
The space $H^1_0(B_1,\omega)$ is the closure of $C^\infty_c(B_1)$ with respect to the norm of $H^1(B_1,\omega)$.
If we fix $g\in C^\infty(\overline{B_1})$, then $H_g^{1}(B_1,\omega) $ denotes the set of all $u\in H^1 (B_1,\omega)$
for which $u-g\in H^1_0(B_1,\omega)$.

It follows by density and H\"older's inequality that
$H^1(B_1,\omega)\subset W^{1,1}(B_1)$, where $W^{1,1}(B_1)$
is the classical Sobolev space of functions $u\in L^1(B_1)$ with $\nabla u\in L^1(B_1)$. In fact,
\begin{equation}\label{eq:embeddinginW11}
\| u \|_{W^{1,1} (B_1)} \le C \|u\|_{H^1(B_1, \omega)}
\end{equation}
for some $C > 0$ depending only on $n$ and $[\omega]_{A_2}$. Similarly $H^1_0(B_1,\omega)\subset W^{1,1}_0(B_1)$.
In particular, if $u\in H^1_g(B_1,\omega)$ then $u=g$ on $\partial\Omega$ in the sense of traces in $W^{1,1}(B_1)$.

The space $H^1(B_R,\omega)$ is defined in a similar way by replacing $B_1$ by $B_R$, $R>0$.
The next result is a direct consequence of \cite[Theorem 1.2]{FKS}, see also \cite[Chapter 15]{HKM}. 

\begin{lemma}[Weighted Poincar\'e inequality]\label{lem:Poincare}
There is a constant $C_{P,\omega}>0$ depending only on $n$ and $[\omega]_{A_2}$ such that,
for any $u\in H_0^1(B_R,\omega)$,
\begin{equation} \label{Poincare ineq}
\int_{B_R} |u|^2 \omega\,dx \le C_{P,\omega}R^2 \int_{B_R} |\nabla u|^2 \omega\,dx.
\end{equation}
\end{lemma}

We say that $h\in H^1(B_1,\omega)$ is a weak solution to the degenerate/singular elliptic equation
\begin{equation}\label{eq:harmonicomega}
\D(\omega(x)\nabla h)=0\qquad\hbox{in}~B_1
\end{equation}
if
$$\int_{B_1}\omega(x)\nabla h\nabla\varphi\,dx=0$$
for all $\varphi\in H^1_0(B_1,\omega)$. By \cite{FKS}, given $g\in H^1(B_1,\omega)$,
there exists a unique weak solution $h\in H^1(B_1,\omega)$ to \eqref{eq:harmonicomega}
such that $h-g\in H^1_0(B_1,\omega)$. We have the following local boundedness
and H\"older regularity estimate that was proved in \cite{FKS}.
 
\begin{theorem}[Harnack inequality and H\"older regularity]\label{continuity of omega-harmonic}
Let $h\in H^1(B_1,\omega)$ be a weak solution to \eqref{eq:harmonicomega}.
If $h\geq0$ in $B_1$, then $h$ satisfies an interior Harnack inequality, that is,
there exists $C_H>0$ depending only on $n$ and $[\omega]_{A_2}$ such that
$$\sup_{B_{1/2}}h\leq C_H\inf_{B_{1/2}}h.$$
As a consequence, $h$ is locally H\"older continuous in $B_1$, that is,
there exist constants $\bar{C}>0$ and $0<\alpha<1$ depending
only on $n$ and $[\omega]_{A_2}$ such that 
$$\| h\|_{C^{0,\alpha}(B_{1/2})} \le\bar{C}\left( \frac{1}{\omega(B_1)} \int_{B_1}|h|^2\omega\, dx \right)^{1/2}.$$
\end{theorem}

%%%%%%%%%%%%%%%%%%%%%%%%%%%%%%%%%%%%%%%%%%%%%%%%%%
\section{The penalized problem with parameter $\varepsilon>0$}\label{general prop sct}
%%%%%%%%%%%%%%%%%%%%%%%%%%%%%%%%%%%%%%%%%%%%%%%%%%

In this section we prove existence, boundedness and local H\"older regularity of solutions
to a penalized formulation with parameter $\varepsilon>0$. Throughout this section we fix
a smooth function $g$ on $\overline{B_1}$
such that $g\geq c>0$ on $\partial B_1$, a parameter $\varepsilon>0$, a weight $\omega\in A_2$
and a constant $0<m<\omega(B_1)$.

Let 
$$f_{\varepsilon}(t) := \frac{1}{\varepsilon} (t - m )^+\qquad\hbox{for}~t\geq0$$
and consider the functional 
\begin{equation}\label{penalized functional}
    J_{\varepsilon} (v,\omega,B_1): = \int_{B_1}|\nabla v|^2\omega\,dx + f_{\varepsilon}(\omega(\{v>0\}\cap B_1)).
\end{equation}
The penalized problem we consider is
\begin{equation}\label{penalized problem}\tag{$P_{\varepsilon}$}
    \min\big\{J_{\varepsilon} (v,\omega, B_1): v\in H_g^1(B_1,\omega)\big\}.
\end{equation}
Clearly, if $v\in H^1_g(B_1)$ then $v^+=\max(v,0)\in H^1_g(B_1)$ and
$J_\varepsilon(v^+,\omega,B_1)\leq J_\varepsilon(v,\omega,B_1)$.

\begin{lemma}[Existence of minimizers for \eqref{penalized problem}]\label{lem:existencepenalized}
There is a minimizer $u\geq0$ to problem \eqref{penalized problem}.
\end{lemma}

\begin{proof}
Given any ball $B_r\subset B_1$, let $v_r\in H^1_g(B_1\setminus B_r,\omega)$ be the unique weak solution to
\begin{equation}\label{function with small positive set}
\begin{cases}
\D(\omega(x)\nabla v)=0&\hbox{in}~B_1\\
v=g&\hbox{on}~\partial B_1\\
v=0&\hbox{on}~\partial B_r.
\end{cases}
\end{equation}
By the maximum principle, $\{v_r>0\}\cap B_1=B_1\setminus\overline{B_r}$.
We extend $v_r$ as $0$ inside of $B_r$, so that $v_r\in H^1_g(B_1,\omega)$.
Then, by choosing $r>0$ sufficiently large, 
$$\omega(\{ v_{r} > 0 \}\cap B_1) = \omega(B_1) - \omega(B_r) \le m.$$
Therefore,
\begin{equation}\label{Bound for J}
0 \le J_{\varepsilon} (v_r,\omega,B_1) \le C_0
\end{equation}
for some constant $C_0>0$ independent of $\varepsilon$ and
$$0\leq \mu:=\inf\big\{J_{\varepsilon} (v,\omega, B_1): v\in H_g^1(B_1,\omega)\big\}\leq C_0.$$
Let $\{v_j\}_{j=1}^\infty \subset H_g^1(B_1,\omega)$ be a minimizing sequence for $J_\varepsilon$.
Since $\mu\leq J_\varepsilon(v_j^+,\omega,B_1)\leq J_\varepsilon(v_j,\omega,B_1)$, we can assume that $v_j\geq0$.
For all $j$ sufficiently large, by \eqref{Bound for J}, 
$$\int_{B_1}  |\nabla v_j|^2\omega\,dx \le C_0.$$
Moreover, by the weighted Poincar\'e inequality in Lemma \ref{lem:Poincare},
$$\int_{B_1}|v_j|^2\omega\,dx\leq 2\int_{B_1}|v_j-g|^2\omega\,dx+2\int_{B_1}|g|^2\omega\,dx\leq C(C_0+1)$$
where $C>0$ is independent of $j$.
Thus, there exists $u\in H^1(B_1,\omega)$ such that, up to a subsequence, $v_j \to u$
weakly in $H^1(B_1,\omega)$. By the compact embedding
$H^1(B_1,\omega)\subset\subset L^2(B_1,\omega)$ and \eqref{eq:embeddinginW11},
up to a further subsequence, $v_{j} \to u$ strongly in $L^2 (B_1,\omega)$ and
a.e.\,in $B_1$. In particular, $u\geq0$ a.e.. Moreover, since $v_j-g$ weakly converges in $H^1_0(B_1,\omega)$, and $H^1_0(B_1,\omega)$ is convex and strongly closed, it follows that $u-g\in H^1_0(B_1,\omega)$, so that $u\in H^1_g(B_1,\omega)$.

Next, since $v_j$ converges in $\omega$-measure to $u\geq0$ in $B_1$, it follows that
$$\omega(\{u\neq0\}\cap B_1)=\omega(\{u>0\}\cap B_1)\leq\liminf_{k\to\infty}\omega(\{v_j>0\}\cap B_1).$$
Therefore, by the weak convergence of $\nabla v_j$ to $\nabla u$ and the fact that
$f_{\varepsilon}$ is a continuous, nondecreasing function,
$$J_{\varepsilon}(u,\omega,B_1) \le \liminf\limits_{j \to \infty} J_{\varepsilon} (v_j,\omega,B_1)=\mu$$
and $u$ attains the minimum in \eqref{penalized problem}. 
\end{proof}

\begin{lemma}[Boundedness of minimizers]\label{lem:boundedness}
If $u$ is the minimizer from Lemma \ref{lem:existencepenalized} then
$$\D(\omega(x)\nabla u)\geq0\qquad\hbox{in the weak sense in}~B_1.$$
In particular, by the maximum principle,
$$0 \le u\le \|g \|_{L^\infty(\partial B_1)}\qquad\hbox{a.e.\,in}~B_1.$$
\end{lemma}

\begin{proof}
For a nonnegative function $\phi\in C_c^{\infty} (B_1)$ and any $t>0$, we have $u-t\phi\in H^1_g(B_1)$, so that,
by the minimality of $u$,
\begin{align*}
        0 &\le \frac{1}{2t} ( J_{\varepsilon}( u  - t\phi,\omega,B_1) - J_\varepsilon(u,\omega,B_1)) \\
        &=-\int_{B_1}\omega\nabla u \cdot \nabla \phi\,dx+\frac{t}{2}\int_{B_1}|\nabla\phi|^2\omega\,dx \\
        &\quad  + \frac{1}{2t} \big[ f_{\varepsilon}
        ( \omega (\{ u - t\phi > 0 \}\cap B_1 ) ) - f_{\varepsilon} ( \omega (\{ u > 0\}\cap B_1 ) ) \big].
\end{align*}
The last term above is nonpositive because $f_\varepsilon$ is nondecreasing.
Thus, by letting $t\to0$,
$$\int_{B_1} \omega \nabla u \cdot \nabla \phi\,dx  \leq 0$$
and $u$ is a subsolution in $B_1$ in the weak sense.
\end{proof}

We next prove that minimizers for \eqref{penalized problem} are H\"older continuous by utilizing the scaling features of the problem and showing that we can normalize the minimizer by considering a modified functional $\tilde{J}_{\varepsilon}$. First, we localize with the following definition. 

\begin{definition}
\label{def:localmin}
Given a ball $B$, we say that $u$ is a \emph{local minimizer in $B$} of a functional of the form
  \[
  J(v,\omega,B): = \int_{B} |\nabla v|^2\omega\,dx +f(\omega(\{ v > 0\}\cap B))
  \]
if for every subdomain $E \subset B$, $u$ minimizes 
  \[
  \int_{E} |\nabla v|^2\omega\,dx +f(\omega(\{ v > 0\}\cap E)+ \omega(\{ u>0\} \cap (B \setminus E)) )
  \]
among all functions $v$ for which $v-u\in H^1_0(E,\omega)$. 
\end{definition}

Note that the term $\omega(\{ u>0\} \cap (B \setminus E))$ is a constant in the minimization in Definition \ref{def:localmin} and that a minimizer $u$ of $J_{\varepsilon} (u,\omega,B_1)$
is a local minimizer in $B_1$. Indeed, if there were a set $E$ and a function $v$ on $E$ for which the above functional is smaller than for $u$, then one could take $w = u$ in $B\setminus E$ and $w = v$ in $E$ to get $J_{\varepsilon}(w, \omega, B_1) < J_{\varepsilon}(u, \omega, B_1)$. By taking $E=B_1$, any local minimizer is a (global) minimizer, however it may not satisfy our imposed boundary condition so we do not use this latter property.

For $x_0 \in B_{1/2}$, $ 0 < r < 1/2$ and $\delta> 0$, we define the rescaled function 
$$\tilde{u}(y) = \frac{u(x_0 + ry) }{\delta}\qquad\hbox{for}~y\in B_1.$$
Recall that the weight $\tilde{\omega}(y):=\omega(x_0 + ry)$ has the same $A_2$ constant than $\omega$.
If we let $\tilde c = \omega(\{u>0\} \cap (B_1 \setminus B_r(x_0)))$ and 
$$\tilde{f}_{\varepsilon}(t):= r^{-n} f_{\varepsilon}(r^n t + \tilde c) = \frac{1}{\varepsilon}\left(t + \frac{\tilde c-m}{r^n} \right) ^+\qquad\hbox{for}~t\geq0$$
then, by changing variables, we get that
\begin{equation}\label{eq:rescaledfunctional}
\begin{aligned}
\int_{B_r(x_0)} |\nabla u|^2 \omega\,dx
&+f_{\varepsilon}(\omega(\{ u > 0 \}\cap B_r(x_0))+\tilde c) \\
&= \int_{B_1} \delta^2 r^{n-2} | \nabla\tilde{u}|^2 \tilde{\omega}\,dy + f_{\varepsilon}( r^n \tilde{\omega} ( \{\tilde{u}> 0 \}\cap B_1 )+\tilde c)  \\
&= \delta^2 r^{n-2} \bigg[ \int_{B_1}  | \nabla \tilde{u}|^2 \tilde{\omega}\,dy
+\delta^{-2} r^2 \tilde{f}_{\varepsilon}(\tilde{\omega}( \{ \tilde{u} > 0 \}\cap B_1 ))\bigg].
\end{aligned}
\end{equation}
Therefore, $\tilde{u}$ is a local minimizer in $B_1$ of the functional 
\begin{equation}\label{eq:Jepstilde}
\tilde{J}_{\varepsilon}(v,\tilde{\omega},B_1) = \int_{B_1}  | \nabla v|^2 \tilde{\omega}\, dx
+ \eta \tilde{f}_{\varepsilon} (\tilde{\omega}( \{ v>0  \} \cap B_1))
\end{equation}
where $ \eta=(r/\delta)^2> 0$, and  $\tilde \omega$ is an $A_2$ weight.

\begin{theorem}\label{reg under smallness}
Let $0<\alpha<1$ be as in Theorem \ref{continuity of omega-harmonic}.
There exist constants $0<\eta_0<1$ and $0<\lambda<1/2$ depending only on $n$, $[\tilde\omega]_{A_2}$
and $\varepsilon$ such that, for any $0<\eta<\eta_0$ and any local minimizer $\tilde u$ of \eqref{eq:Jepstilde} satisfying
$$\frac{1}{\tilde \omega(B_1)}\int_{B_1}|\tilde{u}|^2\tilde{\omega}\,dx \leq 1$$
we have 
$$|\tilde{u} (x) - \tilde{u} (0) | \le C |x|^{\alpha/2}$$
for all $x \in B_{\lambda}$, where $C>0$ depends only $n$, $[\tilde{\omega}]_{A_2}$, $\eta_0$ and $\lambda$. 
\end{theorem}

\begin{proof}
We begin by setting up the constants. Let
$$C_\ast:=\frac{C_{P,\tilde{\omega}}}{4\varepsilon}D$$
where $C_{P,\tilde{\omega}}$ is the Poincar\'e constant in \eqref{Poincare ineq}
and $D$ is the doubling constant in Lemma \ref{lem:propertiesofA2}$(2)$ for $\tilde{\omega}$.
First, we choose a radius $0 < \lambda < 1/4$ such that 
\begin{equation}\label{choice of radius}
4\bar{C}^2(D+1) \lambda^{2\alpha}\le\frac{1}{2} \lambda^{\alpha} 
\end{equation}
where $\bar{C}$ and $\alpha$ are the constants given in Theorem \ref{continuity of omega-harmonic}.
Next, we fix $0<\eta_0<1$ satisfying
\begin{equation}\label{choice of eta}
\eta_0\leq\min\left\{ \frac{1}{C_\ast\lambda^{2(1- \alpha/2)}} , \frac{c_1}{4C_\ast}\lambda^\alpha \left( \frac{|B_{\lambda}|}{|B_{1/2}|}\right)^2 \right\}
\end{equation}
where $c_1 > 0$ is the constant from Lemma \ref{lem:propertiesofA2}$(1)$ for $\tilde{\omega}$. 

We prove that there exists a sequence of real numbers $\{a_j\}_{j\geq0}$ such that 
\begin{equation}\label{a_k decay}\tag*{$(i)_j$}
\frac{1}{\tilde{\omega}(B_{\lambda^j})}\int_{B_{\lambda^j}} |\tilde u - a_j |^2 \tilde \omega \,dx \le \lambda^{j \alpha} 
\end{equation}
and
\begin{equation}\label{a_k convergence}\tag*{$(ii)_j$}
|a_{j} - a_{j-1} | \le \bar{C}\lambda^{(j-1)\frac{\alpha}{2}}
\end{equation}
for all $j \ge 0$, with the convention that $a_0 = a_{-1} = 0$.  From here, the conclusion follows
using standard arguments (see, for instance, \cite{Stinga_2024})
as Campanato's characterization of H\"older spaces also
holds for metric measure spaces with doubling measures (see \cite{Gorka}).

The proof of \ref{a_k decay}--\ref{a_k convergence} is by induction.
Since, by hypothesis, $\tilde u$ is normalized in $L^2(B_1,\tilde{\omega})$ and $a_0 = a_{-1} = 0$,
we get that \ref{a_k decay}--\ref{a_k convergence} is true for $j=0$.
Suppose that \ref{a_k decay}--\ref{a_k convergence} have been proved for $j = 0, 1, \ldots,  k$.
Define $u_k : B_1 \to \mathbb{R}$ by 
$$u_k(x): = \frac{\tilde u(\lambda^k x) - a_k}{\lambda^{\frac{k \alpha}{2}}}.$$
As in \eqref{eq:rescaledfunctional},
$u_k$ is a local minimizer in $B_1$ of 
$$J_{\varepsilon}^k (v,\omega_k,B_1):=\int_{B_1} |\nabla v|^2 \omega_k\,dx
+ \eta\lambda^{2k(1 - \frac{\alpha}{2})} f_{\varepsilon}^k ( \omega_k (\{ v + \lambda^{-k\alpha/2} a_k > 0 \}\cap B_1 ))$$
where $\omega_k(x):=\tilde \omega(\lambda^k x)$ has the same $A_2$ constant as $\tilde\omega$ and $f_{\varepsilon}^k(t) := \lambda^{-kn} \tilde f_{\varepsilon}(\lambda^{kn} t + \tilde c_k)$, with $\tilde c_k = \tilde \omega(\{ \tilde u>0\} \cap (B_1 \setminus B_{\lambda^k}))$.
 
Let $h_k$ be defined as $h_k=u_k$ on $B_1 \setminus B_{1/2}$ and extended into $B_{1/2}$ as the unique weak solution to
\begin{equation}\label{eq:weakformhuk}
\begin{cases}
\D(\omega_k(x)\nabla h) = 0&\hbox{in}~B_{1/2} \\
h = u_k&\hbox{on}~\partial B_{1/2}.
\end{cases}
\end{equation}
By Theorem \ref{continuity of omega-harmonic}, $h_k$ is locally H\"older continuous in $B_{1/2}$.

We claim that 
\begin{equation}\label{eq:ukminushk}
\frac{1}{\omega_k(B_{1/2})}\int_{B_{1/2}}|u_k-h_k|^2\omega_k\,dx\leq C_\ast\eta\lambda^{2k(1-\frac{\alpha}{2})}.
\end{equation}
For this, first notice that $h_k$ is a competitor for $u_k$ so 
\begin{multline*}
   \int_{B_{1/2}} |\nabla u_k |^2 \omega_k\, dx \le \int_{B_{1/2}} |\nabla h_k |^2 \omega_k\, dx \\+
   \eta\lambda^{2k(1 - \frac{\alpha}{2})}\big(f_{\varepsilon}^k(
   \omega_{k}( \{ h_k  > - a_k\lambda^{-k \alpha/2} \}\cap B_1)-f_{\varepsilon}^k(
   \omega_{k}( \{ u_k  > - a_k\lambda^{-k \alpha/2} \}\cap B_1)\big).
   \end{multline*}
Since $f^k_{\varepsilon}$ is Lipschitz with Lipschitz constant $1/\varepsilon$, we get the following, where we use the doubling property of $\omega_k$ (see Lemma \ref{lem:propertiesofA2}$(2)$) in the second line:
   \begin{align*}
   \int_{B_{1/2}} |\nabla u_k |^2 \omega_k\, dx 
  % &\le \int_{B_{1/2}} |\nabla h_k |^2 \omega_k\, dx +\frac{\eta\lambda^{2k(1 - \frac{\alpha}{2})}}{\varepsilon} \big|\omega_{k}( \{ h_k  > - a_k\lambda^{-k \alpha/2} \}\cap B_1)\\
  % & \quad -
  % \omega_{k}( \{ u_k  > - a_k\lambda^{-k \alpha/2} \}\cap B_1)\big|\\
   &\le \int_{B_{1/2}} |\nabla h_k |^2 \omega_k\, dx +\frac{\eta\lambda^{2k(1 - \frac{\alpha}{2})}}{\varepsilon} \omega_k(B_1)
   \\
   &\le \int_{B_{1/2}} |\nabla h_k |^2 \omega_k\, dx +\frac{\eta\lambda^{2k(1 - \frac{\alpha}{2})}}{\varepsilon}
    D\omega_k (B_{1/2}).
\end{align*}
From here, by the Poincar\'e inequality \eqref{Poincare ineq} (note that $C_{P,\omega_k}=C_{P,\tilde{\omega}}$) and
by using $u_k-h_k$ as test function in \eqref{eq:weakformhuk},
\begin{align*}
\int_{B_{1/2}}|u_k-h_k|^2&\omega_k\,dx
\leq \frac{C_{P,\omega_k}}{4}\int_{B_{1/2}} |\nabla(u_k - h_k)|^2 \omega_k\,dx \\
&= \frac{C_{P,\omega_k}}{4}\int_{B_{1/2}}\big(|\nabla u_k|^2 - \nabla u_k \cdot \nabla h_k \big)\omega_k\,dx\\
&\le \frac{C_{P,\omega_k}}{4}\int_{B_{1/2}}\big(|\nabla h_k|^2 - \nabla u_k \nabla h_k\big)\omega_k\,dx
+\frac{C_{P,\omega_k}}{4}\frac{\eta\lambda^{2k(1-\frac{\alpha}{2})}}{\varepsilon}D\omega_k(B_{1/2}) \\
&= C_\ast\eta\lambda^{2k(1-\frac{\alpha}{2})}\omega_k(B_{1/2})
\end{align*}
and \eqref{eq:ukminushk} is proved. Now, since by induction hypothesis and the doubling property we have
$$\int_{B_1}|u_k|^2  \omega_k \,dx \le D\omega_k(B_{1/2}),$$
by applying \eqref{eq:ukminushk}, we can estimate
\begin{equation}\label{average of h_k}
\begin{aligned}
\int_{B_{1/2}}|h_k|^2\omega_k\,dx
&\leq 2\int_{B_{1/2}}|h_k-u_k|^2\omega_k\,dx + 2\int_{B_{1/2}}|u_k|^2\omega_k\,dx \\
&\leq 2(C_\ast\eta\lambda^{2k(1-\frac{\alpha}{2})}+D)\omega_k(B_{1/2}) \\
&\leq 2(D+1)\omega_k(B_{1/2}),
\end{aligned}
\end{equation}
where we used that $\eta \le \eta_0$ and \eqref{choice of eta}.

To prove the induction step, let $a_{k+1}: = a_k + \lambda^{k\frac{\alpha}{2}} h_k(0)$.
Then, by using that $\tilde\omega(B_{\lambda^{k+1}}) =  \lambda^{kn} \omega_k(B_{\lambda}) $
and \eqref{eq:ukminushk}, Theorem \ref{continuity of omega-harmonic}, \eqref{average of h_k}
and Lemma \ref{lem:propertiesofA2}$(1)$,
\begin{align*}
\frac{1}{\tilde{\omega}(B_{\lambda^{k+1}})}\int_{B_{\lambda^{k+1}}}&|u - a_{k+1}|^2\tilde\omega\,dx
%&= \frac{1}{\omega_k(B_\lambda)} \int_{B_{\lambda}} |u(\lambda^k x) - a_{k+1}|^2 \tilde\omega(\lambda^k x)\,dx  \\
= \frac{\lambda^{k\alpha}}{\omega_k(B_\lambda)}\int_{B_{\lambda}}|u_k(x) - h_k (0) |^2\omega_k\,dx  \\
&\le \frac{2\lambda^{k\alpha}}{\omega_k(B_\lambda)}
\Bigg[\int_{B_{\lambda}}  |u_k - h_k|^2 \omega_k\, dx + \int_{B_{\lambda}} |h_k - h_k(0)|^2 \omega_{k} \,dx  \Bigg]  \\
&\le \lambda^{k\alpha}\Bigg[\frac{2C_\ast\eta\lambda^{2k(1 - \frac{\alpha}{2})}\omega_k(B_{1/2})}{
\omega_k (B_{\lambda})} + 4\bar{C}^2 (D+1) \lambda^{2\alpha} \Bigg] \\
&\leq \lambda^{k\alpha}\Bigg[\frac{2C_\ast\eta_0|B_{1/2}|^2}{c_1|B_{\lambda}|^2} + 4\bar{C}^2 (D+1) \lambda^{2\alpha} \Bigg] \\
&\leq \lambda^{(k+1)\alpha}
\end{align*}
where in the last line we used \eqref{choice of radius} and \eqref{choice of eta}.
This proves \ref{a_k decay} for $j = k+1$.  To see that \ref{a_k convergence} holds for $k+1$, we note 
that $|a_{k+1} - a_k| = |h_{k}(0)| \lambda^{k\alpha/2} \le\bar{C} \lambda^{k\alpha/2}$.
\end{proof}
 
\begin{theorem}[Regularity] \label{thm-regularity}
Any minimizer  $u$ for \eqref{penalized problem} is locally H\"older continuous in $B_1$
and, for any compact set $K \subset B_1$,
$$[u]_{C^{\alpha/2} (K)} \le C \|  u\|_{L^2 (B_1,  \omega)}$$
where $0<\alpha<1$ is as in Theorem \ref{continuity of omega-harmonic} and
$C > 0$ depends only on $K$, $n$ and $[\omega]_{A_2}$. 
\end{theorem}

\begin{proof}
Let $\eta_0>0$ be as in Theorem \ref{reg under smallness}.
Fix  $0<\eta<\eta_0$ and $x_0 \in B_1$. Define 
$$\tilde{u}(x) = \frac{u(x_0 + rx)}{\delta}\qquad\hbox{for}~x\in B_1$$
where 
$$r = \min \left( \frac{\operatorname{dist}(x_0, \partial B_1)}{10}, \Bigg(\frac{\eta}{\omega(B_1)}\int_{B_1}|u|^2\omega\, dx\Bigg)^{1/2}  \right)\qquad\hbox{and}\qquad\delta = \frac{r}{\sqrt{\eta}}. 
$$
Then $\tilde{u}$ satisfies the hypotheses of Theorem \ref{reg under smallness}, so that
$|\tilde{u}(y)-\tilde{u}(0) | \le C |y|^{\alpha/2}$
whenever $|y| < \lambda$. By rescaling back to $u$, the conclusion follows.
\end{proof}

\begin{corollary}
If $u\geq0$ is a local minimizer of \eqref{penalized functional} then $\D(\omega(x)\nabla u)=0$ in the
weak sense in the open set $\{ u > 0\} \cap B_1$.
\end{corollary}

\begin{proof}
As a consequence of the local H\"older regularity in Theorem \ref{thm-regularity},
$\{ u > 0 \}$ is an open set in $B_1$. 
By Lemma \ref{lem:boundedness}, $u$ is a subsolution in $B_1$. To check that $u$ is a supersolution in $\{u>0\}$,
consider a nonnegative $\psi\in C^\infty_c(\{u>0\})$ and any $t>0$.
Since $\{u > 0 \} = \{ u + t\psi> 0\}$,
$$0 \le \frac{1}{2t}( J_{\varepsilon}( u  + t\psi,\omega,B_1) - J_\varepsilon(u,\omega,B_1))=\int_{B_1}\omega\nabla u \cdot \nabla \psi\,dx+\frac{t}{2}\int_{B_1}|\nabla\psi|^2\omega\,dx$$
and, by letting $t\to0$,
  \[
  \int_{\{u>0\}}\omega\nabla u\cdot\nabla\psi\,dx\geq0. \qedhere \]
\end{proof}

%%%%%%%%%%%%%%%%%%%%%%%%%%%%%%%%%%%%%%%%%%%%%%%%%%
\section{Behavior  for small $\varepsilon$}\label{small eps sct}
%%%%%%%%%%%%%%%%%%%%%%%%%%%%%%%%%%%%%%%%%%%%%%%%%%

In this section, we prove the existence part of Theorem \ref{thm:main}.
By the results of Section \ref{general prop sct}, the regularity properties of the minimizer
stated in Theorem \ref{thm:main} hold.

The main result of this section is the following.

\begin{theorem}
Let $\omega$ be an $A_2$ weight satisfying either the isoperimetric inequality \eqref{eq:isoperimetricassumption} or 
the lower bound \eqref{eq:lowerboundassumption}.
Then there exists $0 < \varepsilon_*<1$ such that if $u_* \in H_g^1 (B_1, \omega)$ is a nonnegative
minimizer for \eqref{penalized problem} as in Section \ref{general prop sct} then 
$$\omega(\{ u_* > 0 \}\cap B_1 ) = m.$$
In particular, $u_\ast$ is a minimizer for \eqref{weighted prob}.
\end{theorem}

\begin{proof}
Let $u \in H_g^1 (B_1, \omega)$ be a minimizer of \eqref{penalized problem}
as in Section \ref{general prop sct}. 
Without loss of generality, we may assume, up to rescaling, that $\inf_{\partial B_1} g := \gamma \ge 1$.

Suppose that $\omega( \{ u > 0\}\cap B_1 ) < m$. 
Since $u$ is locally H\"older continuous, there exists $x_0 \in \partial\{u > 0\} \cap B_{1}$
and $r>0$ sufficiently small such that 
\begin{equation}\label{positivity on FB}
0<\omega(\{ u > 0\}\cap B_{r/2}(x_0))  \le \omega(B_r(x_0)) < m - \omega (\{ u > 0\} ).
\end{equation}
Next, let $v \in H^1(B_1,\omega)$ such that $\D(\omega(x)\nabla v)=0$
in the weak sense in $B_r(x_0)$ and $v=u$ in $B_1 \setminus B_r(x_0)$. 
As $\{ v > 0\}\cap B_1 = ( \{ v> 0\} \cap B_r(x_0)) \cup(( \{ u > 0\}\cap B_1) \setminus B_r(x_0))$, we get
$$\omega( \{ v > 0\}\cap B_1)  \le  \omega(B_r(x_0)) +  \omega ( (\{ u > 0\}\cap B_1) \setminus B_r(x_0) ) < m. 
$$
Hence, 
\begin{align*}
J_{\varepsilon}(v,\omega,B_1) &= \int_{B_1} | \nabla v|^2 \omega\, dx \\
    &= \int_{B_r(x_0)} |\nabla v|^2 \omega\,dx + \int_{B_1 \setminus B_r(x_0)} |\nabla u|^2 \omega\, dx \\
    &< \int_{B_1} |\nabla u|^2\omega\,  dx
\end{align*}
since $u\geq0$ cannot be a solution to $\D(\omega(x)\nabla u)=0$ in $B_r(x_0)$
by virtue of \eqref{positivity on FB} and the Harnack inequality.
This is a contradiction to the minimality of $u$. Thus,
$$\omega ( \{ u > 0 \}\cap B_1 ) \ge m.$$ 

To prove the conclusion assume, for the sake of contradiction,
that for all $\varepsilon>0$ we have $\omega ( \{ u > 0 \} \cap B_1) > m$. Then, for all $t > 0$ small enough (depending on $u$) there holds
$$\omega( \{ u > t \} \cap B_1) > m.$$
Define
$$u_{t}=
\begin{cases}
\displaystyle\frac{(u - t)^{+}}{1 - t}&\hbox{if}~u \le 1\\
u&\hbox{if}~u\ge1.
\end{cases}$$
Since $u_t \in H_{g}^1 (B_1,\omega)$,
$$J_{\varepsilon} (u, \omega, B_1) \le J_{\varepsilon} (u_t , \omega , B_1)$$
which implies
$$\int_{B_1} \big( |\nabla u|^2 - |\nabla u_{t}|^2 \big) \omega\,dx
\le \frac{1}{\varepsilon} \big( \omega( \{ u_t > 0 \}\cap B_1) - \omega( \{ u > 0 \}\cap B_1) \big).$$
Since 
$$\int_{B_1}\big( |\nabla u|^2 - |\nabla u_{t}|^2 \big)\omega\,dx  =\int_{ \{ 0 < u < t\}} |\nabla u |^2 \omega\,dx
+ \frac{t^2 - 2t}{(1 -t)^2} \int_{\{ t < u < 1\}} |\nabla u|^2  \omega\, dx
$$
and  $\{ u_t > 0 \} = \{ u > t \}$, we obtain 
$$\int_{\{ 0 < u < t \}} |\nabla u|^2 \omega\,dx + \frac{1}{\varepsilon}
\omega\big(\{ 0 < u < t\}\cap B_1\big) \le \frac{2t - t^2}{(1 -t)^2} \int_{B_1} |\nabla u|^2 \omega\,dx.$$
Note that 
$$\int_{B_1} |\nabla u|^2 \omega\,dx \le J_{\varepsilon} (u,\omega,B_1) \le \int_{B_1} |\nabla v_r|^2 \omega\,dx$$
where $v_r$ is the function constructed in \eqref{function with small positive set}. Thus,
\begin{equation}\label{estimate fo u < t}
     \int_{\{ 0 < u < t \}} |\nabla u|^2 \omega\, dx + \frac{1}{\varepsilon} \omega\big( \{ 0 < u < t\}\cap B_1\big) \le \frac{2t - t^2}{(1 -t)^2} C
\end{equation}
where $C > 0$ is a constant independent of $\varepsilon$. Next, we write the left hand side of \eqref{estimate fo u < t} as
$$\int_{\{ 0 < u < t \}} |\nabla u|^2 \omega\,dx + \frac{1}{\varepsilon} \omega\big(\{ 0 < u < t\}\cap B_1\big) 
=\int_{\{ 0 < u < t\}} \big( |\nabla u|^2 + \varepsilon^{-1} \big) \omega\,dx.$$
To estimate this term, recall that $u \in W^{1,1}(B_1)$ is locally H\"older continuous by Theorem \ref{thm-regularity}.
Then we can apply the co-area formula for Sobolev functions from \cite[Theorem 1.4]{maly-swanson-ziemer} to get
\begin{align*}
\int_{\{0<u<t\}}\big(|\nabla u|^2 +& \varepsilon^{-1}\big)\omega\, dx
= \lim_{N \to \infty}\int_{\{0<u<t\} \cap \{1/N<|\nabla u|<N\}}\big(|\nabla u|^2 + \varepsilon^{-1}\big) \omega\, dx  \\
&= \lim_{N \to \infty}\int_{\{0<u<t\} \cap \{1/N<|\nabla u|<N\}} |\nabla u| \bigg(|\nabla u| + \frac{\varepsilon^{-1}}{|\nabla u|} \bigg) \omega\, dx \\
&= \lim_{N \to \infty}\int_0^\infty\int_{\{0<u<t\}\cap\{1/N<|\nabla u|<N\}\cap\{u =s\}} \bigg(|\nabla u| + \frac{\varepsilon^{-1}}{|\nabla u|} \bigg)
\omega\,d\mathcal{H}^{n-1} \, ds \\
&= \int_0^t\int_{\{ u = s\}}\bigg(|\nabla u| + \frac{\varepsilon^{-1}}{|\nabla u|} \bigg)\omega\,d\mathcal{H}^{n-1}\, ds.
\end{align*}
Consequently, and using that $x+\frac{\varepsilon^{-1}}{x}\geq 2 \sqrt{\varepsilon^{-1}}$ for all $x>0$, 
\begin{equation}\label{eq:crucial}
\begin{aligned}
 \frac{2t - t^2}{(1 -t)^2} C&\geq  \int_0^{t} \int_{\{ u = s\}} \bigg( |\nabla u| +\frac{\varepsilon^{-1}}{|\nabla u|} \bigg) \omega\, d\mathcal{H}^{n-1}\, ds \\
&  \ge   2  \sqrt{\varepsilon^{-1}} \int_0^t \int_{\{ u = s\} }  \omega\, d\mathcal{H}^{n-1}\, ds.
\end{aligned}
\end{equation}

Assume first that $\omega$ satisfies the isoperimetric inequality \eqref{eq:isoperimetricassumption}.
Then, from \eqref{eq:crucial},
$$\frac{2t - t^2}{(1 -t)^2} C\ge \frac{2 \sqrt{\varepsilon^{-1}}}{C_0} \int_0^t \omega (\{ u > s \}\cap B_1)^{\frac{n-1}{n}}\, ds.$$
After multiplying by $1/t$ and letting $t \to 0$, by the Lebesgue differentiation theorem, we arrive at
$$CC_0\sqrt{\varepsilon} \ge \omega ( \{ u > 0 \}\cap B_1)^{\frac{n-1}{n}}>m^{\frac{n-1}{n}}.$$
By choosing $\varepsilon>0$ small enough, we reach a contradiction.

Assume next that $\omega$ is bounded below by $\tau$ as in \eqref{eq:lowerboundassumption}.
Then, by applying this assumption to \eqref{eq:crucial}
and the isoperimetric inequality from \cite{maly-swanson-ziemer},
\begin{align*}
 \frac{2t - t^2}{(1 -t)^2} C
&\ge 2\tau \sqrt{\varepsilon^{-1}} \int_0^t \mathcal{H}^{n-1}(\{u=s\}\cap B_1)\, ds \\
&\ge 2\tau \sqrt{\varepsilon^{-1}} C_n\int_0^t |\{u>s\}\cap B_1|^{\frac{n-1}{n}}\,ds.
\end{align*}
After multiplying by $1/t$ and letting $t\to0$, we get
$$\frac{\sqrt{\varepsilon} }{C_n\tau} \ge | \{ u > 0 \} \cap B_1|^{\frac{n-1}{n}}.$$
Since we are assuming that $\omega( \{ u > 0 \}\cap B_1) > m$, by applying
Lemma \ref{lem:propertiesofA2}$(3)$, we can estimate
$$| \{ u > 0 \}\cap B_1| \geq \bigg(\frac{m}{C_2 \omega(B_1)}\bigg)^{\frac{1}{S_D}} |B_1|=:\tilde{m}$$
and obtain
$$\frac{\sqrt{\varepsilon} }{C_n\tau}\geq\tilde{m}^{\frac{n-1}{n}}$$
which again gives a contradiction by choosing $\varepsilon > 0$ sufficiently small.

We have shown that there exists $0<\varepsilon_\ast<1$ such that if $u_\ast\in H^1_g(B_1,\omega)$ is a
nonnegative minimizer for \eqref{penalized problem} then $\omega(\{u_\ast>0\}\cap B_1)=m$.
We claim that $u_\ast$ is a minimizer for \eqref{weighted prob}. Clearly, $u_\ast\in K_0$
and, for any $v\in K_0\subset H^1_g(B_1)$,
$$\int_{B_1}|\nabla u_\ast|^2\omega\,dx\leq J_{\varepsilon_\ast}(u_\ast,\omega,B_1)\leq J_{\varepsilon_\ast}(v,\omega,B_1)
=\int_{B_1}|\nabla v|^2\omega\,dx$$
as claimed.
\end{proof}

%%%%%%%%%%%%%%%%%%%%%%%%%%%%%%%%%%%%%%%%%%%%%%%%%%
%References
%%%%%%%%%%%%%%%%%%%%%%%%%%%%%%%%%%%%%%%%%%%%%%%%%%

\bibliographystyle{plain}
\bibliography{biblio}
\end{document}